\documentclass[reqno,12pt]{amsart}

\usepackage{amsmath,amssymb,graphicx}

\usepackage[latin1]{inputenc}
\usepackage[english]{babel}

\newtheorem{Theorem}{Theorem}[section]

\newtheorem{corollary}[Theorem]{Corollary}
\newtheorem{lemma}[Theorem]{Lemma}

\theoremstyle{definition}
\newtheorem{Remark}[Theorem]{Remark}
\newtheorem{definition}[Theorem]{Definition}

\title[Generating series of cyclically fully commutative elements]{Generating series of cyclically fully commutative elements is rational}

\author[Mathias P\'etr\'eolle]{Mathias P\'etr\'eolle}
\address{Institut Camille Jordan, Universit\'e Claude Bernard Lyon 1,
69622 Villeurbanne Cedex, France}
\email{petreolle@math.univ-lyon1.fr}
\urladdr{http://math.univ-lyon1.fr/{\textasciitilde}petreolle}

\begin{document}
\maketitle

\begin{abstract}
In this paper,  we study the generating function of cyclically fully commutative elements in Coxeter groups, which are elements such that any cyclic shift of theirs reduced decompositions remains a reduced expression of a fully commutative element. By designing a finite state automaton recognizing reduced expressions of cyclically fully commutative elements, we can show that the aforementioned generating series is always rational.

\end{abstract}

\section{Introduction}

Recall that a Coxeter group $W$ is defined by a finite set $S$ of generators which are all involutions and which are subject to \emph{braid relations} of the form $st\cdots=ts\cdots$ where both sides of the equality have the same length, namely $m_{st} \geq 2$. Combinatorics of Coxeter is a very central and studied question, related to various domains in mathematics, such as Temperley--Lieb algebras, representation theory, or geometric topics (see for instance \cite{BB, HUM}).  This combinatorics arises in particular from the properties of \emph{reduced expressions} of an element $w \in W$, which are words $s_1 \ldots s_\ell$ in $S^*$, the free monoid generated by $S$, representing $w$ such that $w$ can not be written as a product with less than $\ell$ generators. 

 Note that the Matsumoto's property ensures us that for a generic element $w \in W$, any two of its reduced decompositions are linked by a series of braid relations. An element of $W$ is said to be \emph{fully commutative} (FC) if any two of its reduced decompositions can be linked by a series of commutation relations of adjacent letters. These elements were introduced and studied independently by Fan \cite{FAN}, Graham \cite{GRA} and Stembridge \cite{STEM1, STEM2, STEM3}. In particular, Stembridge classified the Coxeter groups with a finite number of fully commutative elements and enumerated them in each case. Computing this enumeration is a relevant question, as these elements index a linear basis of (generalized) Temperley--Lieb algebras. More recently, Biagioli--Jouhet--Nadeau \cite{BJN2} extend these results, and enumerated these elements in all finite and affine Coxter groups according to their lengths. Nadeau also proved in \cite{NAD} that the generating function of fully commutative elements is rational in all Coxeter groups.
 
 In this paper, we will focus on a certain subset of fully commutative elements, the \emph{cyclically fully commutative} (CFC) elements. These are elements $w$ for which every cyclic shift of any reduced expression of $w$ is a reduced expression of a FC element (which can be different from $w$).  They were introduced by Boothby \emph{et al.} in \cite{BBEEGM}, where the authors classified the Coxeter groups with a finite number of CFC elements (they showed that they are exactly the groups with a finite number of FC elements) and enumerated them.  These elements were introduced as a class of elements generalizing the Coxeter elements, in that they share certain key combinatorial properties, and in order to study a cyclic version of Matsumoto's theorem.
 
The CFC elements was also enumerated for all finite and affine Coxeter groups by the author in \cite{PET4}, using a new construction, the cylindrical closure. It appears in this work that the generating function of CFC elements is in all affine and finite Coxeter groups a rational function from which the coefficients are ultimately periodic. This opens the question about the nature of this generating series for other Coxeter groups. This is the subject of this article, in which we give an affirmative answer for the rationality of the generating function. We write $W^{CFC}_k$ the set of all CFC elements of length $k$, and we fix $W^{CFC}(x):= \sum_{k \geq 0} |W^{CFC}_k|x^k$. Our main contribution is the following theorem.

\begin{Theorem}\label{thmintro}
Let $(W,S)$ be a Coxeter system. The generating series $W^{CFC}(x)$ is rational.
\end{Theorem}

There are many results about the rationality of generating series for various subsets of $W$, including for example the length generating series of $W$ itself (see \cite[Section 7.1]{BB}), or the length generating series of  fully commutative elements (see \cite[Theorem 1.1]{NAD}).

Unlike in the case of $W$, there does not seem to be a (simple) recursive decomposition of $W^{CFC}$ which behaves well with respect to length, which will be sufficient to prove the rationality of $W^{CFC}(x)$. Indeed, an equivalent formulation of this rationality is that the sequence $|W^{CFC}_k|$ satisfies a linear recurrence with constant coefficients for $k$ large enough. That is why, to prove our Theorem~\ref{thmintro}, we will use an indirect approach through the theory of \emph{finite state automata}, like Nadeau did for fully commutative elements in \cite{NAD}. In Section~\ref{preuverecauto}, we build an automaton recognizing a language with generating function $W^{CFC}(x)$, which induces the rationality of $W^{CFC}(x)$. To do this, we begin by designing an explicit automaton recognizing the set of all reduced expressions of CFC elements in $W$. This design is essentially the core of this paper. Next, we go from reduced expressions to the elements by using a lexicographical order through a result of Brink and Howlett, which complete the proof of Theorem~\ref{thmintro}.

We next deduce some corollaries from this theorem, in particular we re-obtain a result due to Nadeau.

\medskip
 
This paper is organized as follows. We recall in Section~\ref{section2} some definitions. More precisely, we first recall some notions on Coxeter groups and secondly we present the basic notions about the theory of finite state automata. Section~\ref{preuverecauto} is devoted to the proof of Theorem~\ref{thmintro}.  In its first part, we design the aforementioned automaton recognizing reduced expressions of CFC elements, and prove its correctness, which is the crucial point of this paper. In Paragraph~\ref{3.2}, we achieve the proof of Theorem~\ref{thmintro}, and deduce some corollaries.  Section~\ref{section4} ends this paper with some open questions and perspectives.

\section{Cyclically fully commutative elements and automata theory}\label{section2}

\subsection{Coxeter groups} Here we recall some useful notions about Coxeter groups.

 Let $W$ be a Coxeter group with finite generating set $S$ and Coxeter matrix $M=(m_{st})_{s,t \in S}$. Recall (see \cite{BB}) that this notation means that the defining relations between generators are of the form $(st)^{m_{st}}=1$ for $m_{st} \neq \infty$, where the matrix $M$ is symmetric  with $m_{ss}= 1$ and  $m_{st} \in \{ 2, 3, \ldots\} \cup \{  \infty \}$. The pair ($W, S$) is called a \textit{Coxeter system}. We can write, for each pair $(s,t)$ of distinct generators, the relations as $sts\cdots =tst \cdots$, each side having length $m_{st}$; these are usually called \textit{braid relations} when $m_{st}\geq 3$. When $m_{st}=2$, this is a \textit{commutation relation} or a \textit{short braid relation}. We define the Coxeter diagram $\Gamma$ associated with ($W,S$) as the graph with vertex set $S$, and  edges labelled $m_{st}$ between $s$ and $t$ for each $m_{st} \geq 3$. As is customary, edge labels equal to $3$ are usually omitted.
 
 We define $S^*$ the free monoid generated by $S$, and we call the elements of $S^*$ \emph{expressions} or \emph{words} (we denote by $\varepsilon$ the empty word). Recall that by definition, each element of $W$ can be represented by a word $\bf w$ in  $S^*$, but this word is, in general, not unique. For clarity, in this paper, we always write $w$ for elements in $W$, and $\bf w$ for expressions in $S^*$. If a word $\bf w$ in $S^*$ is equal to $w$ when considered as an element of $W$, we say that $\bf w$ is an expression for $w$, or equivalently, that the corresponding element for $\bf w$ is $w$.  For $w \in W$, we define the \emph{Coxeter length $\ell(w)$} as the minimal length $k$ of any expression ${\bf w}=s_1\ldots s_k \in S^*$ such that ${\bf w}$ is an expression for $w$. An expression $\bf w$ of $w$ is \emph{reduced} if it has minimal length. The set of all reduced expressions of $w$ is denoted by $R(w)$.

A classical result in Coxeter groups theory, known as Matsumoto's theorem (see \cite[Theorem 3.3.1]{BB}) states that, given two words in $R(w)$, one can always go from one to the other by using only braid relations. An element $w \in W$ is \emph{fully commutative} if any expression in $R(w)$ can be obtained from any other one by using only commutation relations (we will often abbreviate the term fully commutative by FC). These elements were characterized by Stembridge \cite[Proposition 2.1]{STEM1} as elements $w$ such that no word in $R(w)$ contains a braid word. Equivalently, it means that a reduced expression $\bf w$ is an expression for a FC element if we can not go by using commutation relations from $ \bf w$ to an expression containing a braid word. We say that two words are \emph{commutation equivalent} if one can go from one to the other by using only commutation relations. 

Given a word ${\bf w }= s_1 \ldots s_n \in S^*$, the \emph{left cyclic shift} of $\bf w$ is $s_2s_3\ldots s_ns_1$. A \emph{cyclic shift} of $\bf w$ is $\bf w$ itself or $s_ks_{k+1}\ldots s_{k-1}$ for $k \in \{2, \ldots n\}$.

\begin{definition}An element $w\in W$ is \emph{cyclically fully commutative}  if every cyclic shift of any expression in $R(w)$ is a reduced expression for a fully commutative element (which can be different from $w$). 

\end{definition}

We will often write CFC for cyclically fully commutative in the rest of the paper. We denote by $W^{CFC}$ the set of CFC elements of $W$, and if $k$ is an integer, $W^{CFC}_k$ denotes the set of CFC elements of length $k$, and we fix \begin{equation}W^{CFC}(x):=\sum_{k \geq 0} |W^{CFC}_k| x^k = \sum_{w \in W} x^{\ell(w)},
\end{equation} 
the generating function of CFC elements according to their Coxeter length.

\subsection{Finite state automata}

We refer to \cite{HOP} for the general theory of finite state automata. Here, we only need the basic notions of this theory.

Let $S$ be a finite set, that we call the \emph{alphabet}, and, as above, let $S^*$ be the free monoid generated by $S$. A \emph{finite, deterministic, complete automaton} $\mathcal{A}$ on the alphabet $S$ is a 5-tuple $(S,Q, \delta, i_0, F)$ such that:
\begin{itemize}
\item $S$ is the alphabet, 
\item $Q$ is a finite set of \emph{states},
\item $\delta$ is a \emph{transition function} $\delta : Q \times S \rightarrow Q$,
\item $i_0$ is a state, called the \emph{initial state},
\item $F$ is a subset of $Q$, containing the \emph{final states}.
\end{itemize}
The function $\delta$ can be extended to a function $\bar{\delta} : Q \times S^* \rightarrow Q$, by induction on the length of the words, by defining for $q\in Q$, $s \in S$ and ${\bf w}\in S^*$, $\bar{\delta}(q, \varepsilon)=q$,  and  $\bar{\delta}(q, s{\bf w})=\bar{\delta}(\delta(q,s) , {\bf w})$. In general, we represent an automaton $\mathcal{A}$ by a labelled multi-graph with vertex set $Q$, oriented and labelled edges  $q \stackrel{s}{\rightarrow} q' $ whenever $\delta(q,s)= q'$, the initial state being pointed by an arrow and the final states being doubled-circled.

\begin{figure}[h!]
\includegraphics[scale=1]{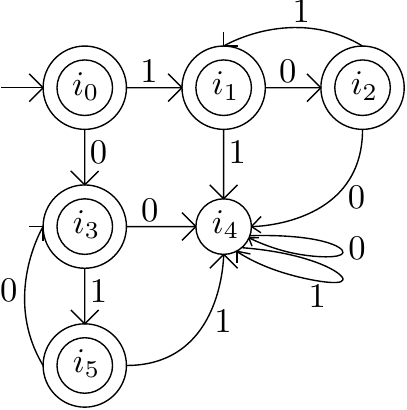}
\caption{Automaton on the alphabet $S=\{0,1\}$.}
\end{figure}

The language $L(\mathcal{A})$ \emph{recognized} by the automaton $\mathcal{A}$ is the set of words ${\bf w}\in S^*$ such that $\bar{\delta}(i_0, {\bf w}) \in F$.  We also say that a such word is \emph{recognized} in this case. In terms of graph, it means that starting from the state $i_0$,  if we read the word $\bf w$ letter by letter, moving each time along the edge labelled by the letter, we end on a vertex belonging to $F$.

Languages which can be recognized by a finite state automaton are called \emph{regular}.  Regular languages are closed under complementation and intersection \cite{HOP}. Moreover, a famous theorem of Kleene ensures us that regular languages coincide with rational languages (languages which can be described by a rational expression). They also satisfies the following theorem.

\begin{Theorem}[{\cite{HOP}}]\label{kleene}
The length generating function of a regular language is a rational series.
\end{Theorem}

\section{Rationality of the generating function of CFC elements}\label{preuverecauto}

This section is devoted to the proof of Theorem~\ref{thmintro} and to its consequences.

\subsection{Rationality of generating function of reduced expression}
The aim of this section is to prove the main following result.
 \begin{Theorem}\label{thmautomate}
 Let $(W,S)$ be a Coxeter system. The set $L$ of reduced expressions of cyclically fully commutative elements is a regular language on the alphabet $S$.
 \end{Theorem}

Let $(W,S)$ be a Coxeter system. To show that the set $L$ of reduced expressions of CFC elements is a regular language on the alphabet $S$,  is is sufficient to show that there exists an automaton $\mathcal{A}$ which recognizes the language $L$. We here explicitly build a such automaton. The proof is divided in three parts: the first describes the states of the automaton, the second indicates how to construct the transition function $\delta$ and the third shows the property of this automaton. The principal idea is to build an automaton which reads only words which are reduced expressions of FC elements (and send other words to a sink state). Moreover, we encode in the states the necessary informations to determine whether the word corresponds to a CFC element or not, and these informations allow us to determine which states are  final.
\medskip

{\bf Description of the states}

The alphabet of the automaton is $S$.  We construct the set of states $Q$ successively, state by state. A state  $q \in Q$ is formally defined as \begin{equation*}
q=(E_q, E_q', (CC_{st,q},IC_{st,q}, B_{st,q}, B_{ts,q})_{(s,t) \in S^2}),
\end{equation*}
where $E_q$ is a subset of $S$, $E_q'$ is a subset of $S^2 \times \{0,1\}$, and for each pair $(s,t) \in S^2$, $CC_{st,q}$ and $IC_{st,q}$ are words on the alphabet $\{s,t, \underline{s}, \underline{t}\}$ of length smaller than $m_{s,t,}$ and both $B_{st,q}$ and $B_{ts,q}$ are boolean. According to this definition, there is only a finite number of states. We now indicate also in an informal manner the informations which are encoded by this state to help the comprehension of this construction of the automaton, however the precise statements of the properties of this automaton are in the last part of the proof:
\begin{itemize}
\item the set $E_q \subset S$ of possible exit letters. It indicates, if a letter $s$ belongs to $E_q$, that if we read  the letter $s$ from this state $q$, the word that we have already read (including the letter $s$) is reduced. If a letter $s$ does not belong to $E_q$, reading the letter $s$ from this state $q$ implies that the word we have already read (including the letter $s$) is not reduced.

\item The set $E'_q$  of pairs of letters $(s,t,0)$ or $(s,t,1)  \in S^2 \times \{0,1\}$   such that if we read the letter $s$ from $q$ , we create a chain  $st\ldots$ or $ts\ldots$ with length $m_{st}$. A such pair is \emph{underlined} if its third component is $1$ and \emph{non underlined} otherwise. In what follows, we will use only this notion of underlined pair, and not the formal definition with a third boolean component. The meaning of this underlined pair will be explained below.

\item For each pair of generators $(s,t) \in S^2$ such that $\infty>m_{st} \geq 3$, $CC_{st,q}$ and $IC_{st,q}$ are two words with letters $s$,  $t$, $\underline{s}$ and $\underline{t}$. In what follows, we will refer to these two words as character chains. The explanation is that we prefer this denomination to avoid the possible confusion between these words and the word reading by the automaton. Moreover, these words are the analogues of the chains defined in the heap of the word reading by the automaton (see for instance \cite{PET4}). 

The first chain is denoted by $CC_{st,q}$ ($CC$ stands for ``current chain''), and indicates the longest chain of type $sts\ldots$ or $tst\ldots$ (with letters eventually underlined) that we read in the word since the last time we encounter in the reading a generator that does not commute with $s$ or $t$.  (In what follows, we identify the chains $CC_{st,q}$  and $CC_{ts,q}$). 

The second chain is denoted by  $IC_{st,q}$  and indicates the longest chain of type $sts\ldots$ or $tst\ldots$ that we have read in the beginning of the word, until we encounter a generator that does not commute with $s$ or $t$.  (Here, $IC$ stands for ``initial chain''). Notice that in this word, the letters are never underlined. 

The two booleans are denoted by $B_{st,q}$ and $B_{ts,q}$; they indicate that if we can always in the construction of the next states add a letter  $s$ (respectively $t$)  to the chain $IC_{st}$, \emph{i.e} we did not already read in the word a generator that does not commute with  $s$ (respectively $t$).
\end{itemize}

\medskip

{\bf Building of the automaton}

We now describe the building of the states  and the transition function of the automaton, which will be an iterative process. Starting with the initial state and a sink, we construct all the states which can be reach by a transition from this initial state through a procedure described below, and we iterate this process until we create all the states of the automaton. 
 
We start from an automaton with two different states. The first is the initial state $i_0$  with set $E_{i_0}$ equal to $S$, the set $E_{i_0}'$ is empty, and for all pairs of generators $(s,t)$ with $m_{st}\geq 3$, the chains $CC_{st,i_0}$ and $CC_{st,i_0}$ are empty and the booleans $B_{st,q}$ and $B_{ts,q}$ are equal to $True$ (this means that from the initial state, we can read  any letter, that we have currently read neither current chains nor initial chains, and that this previous chains can be completed). The second state is a sink, denoted by $P$; we do not define specific sets or chains for this state, all the transitions starting from this state go to this state.
 
 Next, for all states $q$ of the automaton, except for the state $P$, we apply the following process.
\begin{itemize}
\item For each generator $s$ in $S$ which does not belong to $E_q$ or which is equal to the first component of a pair in $E'_q$ (this first component can eventually be underlined), we add an edge starting from $q$ and going to the sink state $P$ and with label $s$ (equivalently, we fix $\delta(q,s)=P$). This means that we read a not reduced word or a word which contains up to commutation an alternating chain  $sts\ldots$ of length $m_{st}$, respectively.

\item For any other generator $s$, we create a new state $r$, defined as follow.  

The set $E_r$ is equal to the set $E_q$ without $s$ and in which we add all the generators in $S$ which do not commute with $s$.

For all  generators $t$ in $S$  which do not commute with  $s$, two cases can occurs according to the value of $B_{st,q}$, these cases are summarized in Figure~\ref{figrecap}, below.   If $B_{st, q} = True$, then the chain  $IC_{st,r}$ is equal to the concatenation of  $IC_{st,q}$ and $s$ and the chain  $CC_{st,r}$ is equal to the concatenation of  $CC_{st,q}$ and $s$, with $s$ underlined. If $B_{ts,q}= True$, we set $B_{st,r}= B_{ts,r}:= True$, else we set $B_{st,r}=B_{ts,r}:=False$. If $B_{st, q} = False$, then $IC_{st,r}=IC_{st,q}$ and the chain $CC_{st,r}$ is equal to the concatenation of $CC_{st,q}$ and $s$, with $s$ not underlined. We set $B_{st,r}=B_{ts,r}= False$.

\begin{figure}[!h]
\includegraphics[scale=1]{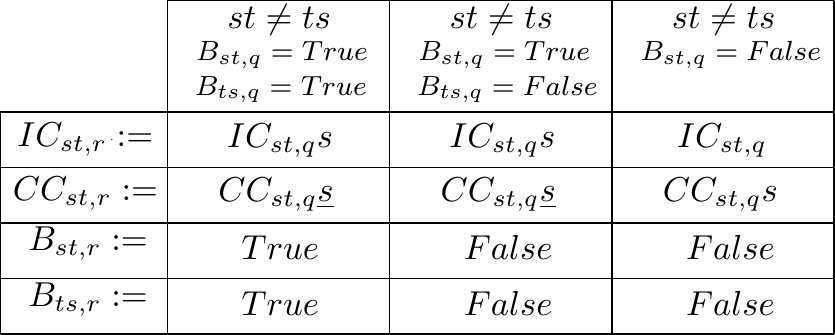}
\caption{\label{figrecap}Description of $r$. }
\end{figure}
 
  If neither  $t$ nor $u$ commute with $s$ then we set that $CC_{tu,r}$ is the empty chain. We fix  $B_{tu,r}:= False$, $B_{ut,r}:= False$, and $IC_{tu,r}:=IC_{tu,q}$.

  If $t$ and $u$ commute with $s$, then the chain $CC_{tu,r}$ is equal to the chain $CC_{tu,q}$. We fix $B_{tu,r}:=B_{tu,q}$, $B_{ut,r}:=B_{ut,q}$, and $IC_{tu,r}:=IC_{tu,q}$.  
 
 If $t$ do not commute with $s$ and  $u$ commute with $s$, many cases must be considered according to the last letter  of the chain  $CC_{tu,q}$. If the chain  $CC_{tu,q}$ ends by $t$ (underlined or not) then $ CC_{tu,r}$ is empty, and we fix $B_{tu,r} := False $, $B_{ut,r}:=B_{ut,q}$, and $IC_{tu,r}:= IC_{tu,q}$. If the chain $CC_{tu,q}$ ends with $u$ underlined, then $ CC_{tu,r}$ is equal to $u$ underlined, and we fix $B_{tu,r} := False $, $B_{ut,r}:=False$, and $IC_{tu,r}:= IC_{tu,q}$. If the chain $CC_{tu,q}$ ends by $u$ not underlined, then $ CC_{tu,r}$ is equal to  $u$ not underlined, and we fix $B_{tu,r} := False $, $B_{ut,r}:=False$, and $IC_{tu,r}:= IC_{tu,q}$.
 
All these previous cases are summarized in Figure~\ref{figrecap2}, below. 
\begin{figure}[!h]
\includegraphics[scale=1]{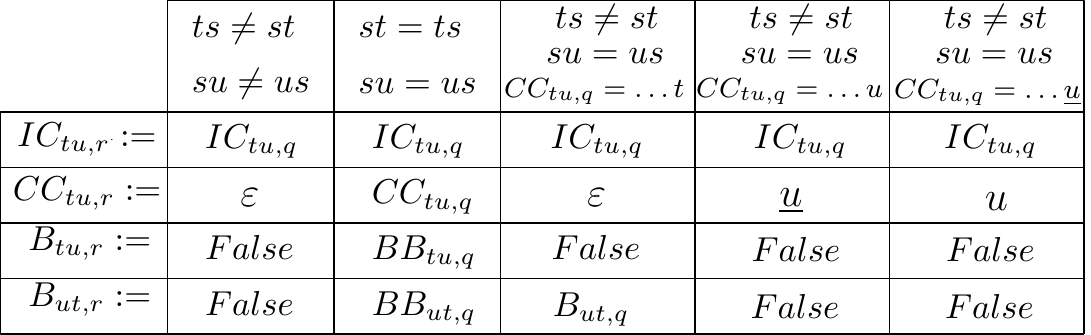}
\caption{\label{figrecap2}Description of $r$. }
\end{figure} 
 
The set $E'_r$ is equal to the set $E'_q$ with two modifications: we delete all the pairs for which the first component does not commute with $s$;  and we add all the pairs $(t,s)$ such that $t$ is a generator which does not commute with $s$ and such that  $CC_{st,r}$ has length $m_{st}-1$. In the case where the letter $s$ which has been added to $CC_{st,r}$ is underlined, then the pair   $(t,s)$ is also underlined in $E'_r$.

We now end the description of the state $r$. If $r$ is equal to an already build state $r'$  in the automaton (which means that all the components describing $r$ and $r'$ are equal, including underlines), then we delete the state $r$ and we add an edge starting from $q$ and going to $r'$ labelled by $s$. In other terms, we fix $\delta(q,s)=r'$. Else, we fix $\delta(q,s)=r$.

\end{itemize}

This iterative process is finite as we can only create a finite number of states. Indeed, there is a finite number of sets  $E_q$ and $E_q'$, there is a finite number of booleans  $B_{st}$, which can only take two different values and there is a finite number of chains $CC_{st}$ and $IC_{st}$, which are all of length at most  $m_{st}-1$ by construction of the automaton.

To end the construction of the automaton, we only need now to define the set of final states. These are the states for which the following conditions are both satisfied:
\begin{itemize}

 \item For all pairs of generators  $(s,t)$ with $m_{st}\geq 3$, the concatenation of the chain $CC_{st,q}$, in which we delete all the underlined letters, with the chain  $IC_{st,q}$ has length smaller or equal to $m_{st}-1$ and does not contain factors $ss$ or $tt$ ;
 \item If a pair $(s,t)$ belongs to $E'_q$ and is not underlined, then the initial chain $IC_{st,q}$ starts with $ t $ or is empty.
 \end{itemize}
 
\medskip

{\bf Properties of the automaton}

The automaton satisfies some properties that we now show.  The following lemma establish the property of the states, as it was mentioned informally in their definitions.
\begin{lemma}\label{lemme1}Let  $r$ be a state of the automaton different from the sink, and let ${\bf w}$ be a word in $S^*$ such that $\bar{\delta}(i_0,{\bf w})=r$. Then, the following properties holds:

\begin{itemize}

\item[($i$)] in $CC_{st,r}$, the underlined letters are all in the beginning of the chain. A letter $s$ is underlined in  $CC_{st,r}$  if and only if during the construction of the state $r$, the letter $s$ was added both to $CC_{st,r}$ and to $IC_{st,r}$. No letter is underlined in $IC_{st,r}$.

\item[($ii$)] The word ${\bf w}$ is a reduced expression, and  $E_q$ is the set of the letters in $S$ such that the word ${\bf w}s$ is also reduced.

\item[($iii$)] Let  $(s,t)$ be a pair of non-commuting generators. The word ${\bf w}$ is commutation equivalent to a word of the form  = ${\bf w}_1CC_{st,r}$ with ${\bf w}_1 \in S^*$, and $CC_{st,r}$ is maximal in the following sense: there is no chain  $C$ consisting of letters $s$ and $t$ with length greater than $CC_{st,r}$ such that ${\bf w}$ is commutation equivalent to ${\bf w}_1' C$ with ${\bf w}_1' \in S^*$.

\item[($iv$)]Let  $(s,t)$ be a pair of non-commuting generators. If we denote by   $D$ the chain $CC_{st,r}$ in which we delete the underlined letters, then the word   ${\bf w}$  is commutation equivalent to a word $IC_{st,q} {\bf w}_2 D$, where ${\bf w}_2$ is a word in $S^*$. This decomposition is maximal, in the sense that if the word  ${\bf w}$ is commutation equivalent to a word ${\bf w}_3{\bf w}_4{\bf w}_5$, with ${\bf w}_3$ and ${\bf w}_4$ two words with only  $s$ and $t$ as letters, then $|{\bf w}_5{\bf w}_3|\leq |DIC_{st,r}|$.

\item[($v$)] The set $E'_r$ contains the pair $(s,t)$ if and only if the word ${\bf w}s$ is commutation equivalent to a word  which contains an alternating chain of letters $t$ and $s$ with length $m_{st}$. The pair $(s,t)$ is underlined if and only if the (induced) alternating chain of letters  $s$ and $t$ with length $m_{st}-1$ which occurs up to commutation in  ${\bf w}$ is also the initial chain $IC_{st,r}$.

\end{itemize}
\end{lemma}
\medskip

\begin{proof}

Property $(i)$ follows immediately from the construction of the automaton: one adds underlined letters only if the chain is empty or after underlined letters.
 
Properties  $(ii),~(iii),~(iv)$, and $(v)$ are shown by induction on the length of the word ${\bf w}$. These four properties are satisfied if  ${\bf w}$ is the empty word. If $\bf w$ is not empty, we write ${\bf w}={\bf w}'u$, with ${\bf w}'$ a word in $S^*$ and $u$ a generator. We denote by $q:= \bar{\delta}(i_0, {\bf w}')$, by definition of $\bar{\delta}$, we also have $\delta(q,u)=r$. This is illustrated in Figure~\ref{figauto2}, below.

\begin{figure}[!h]
\includegraphics[scale=1]{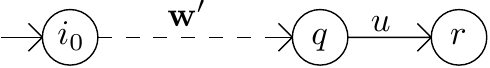}
\caption{\label{figauto2}Notations used in the proof.}
\end{figure}

We first prove $(ii)$. By induction hypothesis, the set $E_q$ contains exactly the letters $s$ such that  the word  ${\bf w}'s$ is reduced.  As  $r$ is not the sink by assumption, by construction of the transitions of the state $q$, the letter $u$ belongs to $E_q$. So, the word ${\bf w}={\bf w}'u$ is reduced. Moreover, the set of letters   $t$ such that ${\bf w}t$ is reduced is equal to  $E_q$ without $u$ and to which we add all the generators which do not commute with $u$. That is exactly the definition of the set $E_r$. This proves $(ii)$.

\medskip

Let us prove now $(iii)$. Let $(s,t)$ be a pair of non-commuting generators. The induction hypothesis ensures us that  ${\bf w}'$ is commutation equivalent to a word ${\bf w}_1'CC_{st,q}$, and that this decomposition is maximal. 

If $u$ is a generator different from $s$ and $t$, and $u$ commute with both $s$ and $t$, then ${\bf w}={\bf w}'u$ is commutation equivalent to ${\bf w}_1CC_{st,q}$ if we set ${\bf w}_1={\bf w}_1'u$, and this decomposition is maximal. Moreover, by construction of the automaton, we have $CC_{st,q}=CC_{st,r}$, which concludes the proof of $(iii)$ in this case.
 
 If  $u$ is a generator different from $s$ and $t$, and  $u$ commute neither with $s$ nor with $t$, then if  ${\bf w}$ is commutation equivalent to a word $ {\bf w}_1C$ where the word $C$ has only  $s$ and $t$ as letters, we have $C= \varepsilon$. By construction of the automaton, in this case we also have $CC_{st,r}= \varepsilon$.
 
 If $u$ is a generator different from $s$ and $t$, such that $u$ commute with $s$ but does not commute with $t$, two cases can occur. If $CC_{st, q}$ ends by a letter $s$, then ${\bf w}$ is commutation equivalent to ${\bf w}_1'CC_{st,q}u$, and so   ${\bf w}$ is commutation equivalent to ${\bf w}_1s$ with ${\bf w}_1$ a word ending by $u$; and this last decomposition is maximal. By construction of the automaton, in this case we have $CC_{st,r}=s$. If $CC_{st, q}$ ends by $t$, then ${\bf w}$ is commutation equivalent to  ${\bf w}_1'CC_{st,q}u$. The longest word $C$ with $s$ and $t$ as letters such that we can decompose ${\bf w}={\bf w}_1C$ is the empty word, and we also have in this case $CC_{st,r}=\varepsilon$.
 
If $u$ is equal to $s$, then necessarily the chain $CC_{st,q}$ is empty or ends by $t$. Indeed,  ${\bf w}$ is commutation equivalent to ${\bf w}_1'CC_{st,q}s$. If $CC_{st,q}$ does not end by $t$ and is not empty, ${\bf w}$ is not reduced, which is a contradiction with property $(ii)$. The facts that  ${\bf w}$ is commutation equivalent to ${\bf w}_1'CC_{st,q}s$, that this decomposition is maximal by induction hypothesis , and that we have by construction of the automaton in this case $CC_{st,r}=CC_{st,q}s$ achieve the proof of $(iii)$.
 
\medskip 
 
The proof of $(iv)$ is very similar to the proof of  $(iii)$, so we do not give the details here. We simply notice that the underlined letters appears here to avoid that we use a same letter in both the initial $IC_{st,r}$ and the current $CC_{st,r}$  chains, in the cases where all the letters of the word commute with $s$ and $t$.

\medskip
 
Now we prove $(v)$. Assume that the pair $(s,t)$ belongs to $E'_r$. Two cases can occur, according to the construction of the automata. In the first case, the pair $(s,t)$ belongs also to $E'_q$. By construction of $E_r$, we can deduce that $u$ commutes with $s$. By induction hypothesis, the word ${\bf w}'s$ contains up to commutation an alternating chain of letters  $s$ and $t$ with length $m_{st}$. This is also the case of the word ${\bf w}'us$ as $u$ commute with $s$. In the second case, the pair $(s,t)$ does not belong to $E_q$. This implies by construction of $E_r$ that the pair $(s,t)$ is such that $CC_{st,q}$ is of length $m_{st}-1$ and ends by $t$, and then the property $(iii)$ ensures us that ${\bf w}s$ is commutation equivalent to ${\bf w}_1CC_{st,q}s$. The converse sense of $(v)$ can be show in a similar way.\end{proof}

We now focus on the properties of the words read by the automaton.

\begin{lemma}\label{lemme2}
Let  $r$ be a state of the automaton different from the sink, and let ${\bf w}$ be a word in $S^*$ such that $\bar{\delta}(i_0,{\bf w})=r$. Then the following property are satisfied:

\begin{itemize}

\item[($vi$)] the word ${\bf w}$ corresponds to a fully commutative element.

\item[($vii$)] If $r \in F$, then ${\bf w}$ corresponds to a cyclically fully commutative element. If $r \notin F$, then ${\bf w}$ does not correspond to a CFC element.
\item[($viii$)] Moreover, if ${\bf w}$ is a word in $S^*$ such that $\bar{\delta}(i_0,{\bf w})=P$, then either  ${\bf w}$ is not reduced or ${\bf w}$ does not correspond to a FC element.

\end{itemize}
\end{lemma}

\begin{proof}

We prove  $(vi)$ by induction  on the length of the word $\bf w$, and we use the same notations as in the proof of Lemma~\ref{lemme1} (see Figure~\ref{figauto2}). Assume that ${\bf w}$ is not fully commutative. As ${\bf w}'$ is fully commutative by induction hypothesis, two cases can occurs. First, ${\bf w}$ could be not reduced, which is impossible according to property $(ii)$ of the previous lemma. Second, the word  ${\bf w}={\bf w}'u$ could be commutation equivalent to a word containing an alternating chain of letters  $t$ and $u$ with length $m_{tu}$. This implies that $u $ belongs to $E'_q$ by property $(v)$. But in this case, by construction of the automaton, the transition starting from $q$ with letter $u$ goes to the sink, and not in the state $r$, which is a contradiction. 
 
\medskip

Now we show $(vii)$. Assume that $r$  is a final state. By property $(i)$ of the previous lemma, we know that the word  ${\bf w}$ corresponds to a fully commutative element.  The only reason which can prevent  ${\bf w}$ from corresponding to a CFC element is the existence of a cyclic shift of  ${\bf w}$ which joints letters from the beginning and from the end of the word $\bf w$ to create a word which is not reduced or does not correspond to a FC element. The restrictions that we fix on initial chains of final states ensure us that this can not happen.

 Indeed, the restriction on pairs belonging to $E'_r$ ensures us that we can not complete through a cyclic shift an alternating chain of length  $m_{st}-1$ of letters $s$ and $t$ ending by a $s$ and following by generators which does not commute with $s$ (which implies that $CC_{st,r}$ is empty); with a generator $t$ belonging to the beginning of the word.  
 
 By the property  $(iv)$ of the previous lemma (and with its notations), we know that for all pairs of not-commuting generators $(s,t)$, the word ${\bf w}$ is commutation equivalent to $IC_{st,r}{\bf w}_2 D$, and this decomposition is maximal. By definition of final states of the automaton, we also have  $|DIC_{st,r}|\leq m_{st}-1$. This maximality ensures us that we can not use a cyclic shift to create a chain  $stst\ldots $ of length $m_{st}$. The condition on the absence of factors $ss$ or $tt$ in $DIC_{st,r}$ ensures us that we can not obtain through cyclic shift a not reduced word. So, the word ${\bf w}$ corresponds to a CFC element if $r$ is a final state. Conversely, if the state $r$ is not final, one of the conditions defining final states is not satisfied, and we verify immediately that a cyclic shift of  ${\bf w}$ is not fully commutative.
 
\medskip
 
Finally, property $(viii)$ follows directly by checking that, by construction, the only way to go to the sink is either to read from a state $q$ a letter which does  not belong to  $E_q$ or to read from a state $q$ a letter $s$ which is the first component of a pair belonging to  $E'_q$. In the first case, by property $(ii)$ of the previous lemma, the word ${\bf w}$ is not reduced, in the second case, ${\bf w}$ is not FC according to property $(v)$.\end{proof}

Now we can achieve the proof of Theorem~\ref{thmautomate}.

\begin{proof}[Proof of Theorem~\ref{thmautomate}]
It remains, to conclude the proof, to notice that properties ($ii$) of Lemma~\ref{lemme1} and ($vii$) of Lemma~\ref{lemme2} ensures us that the so-defined automaton recognizes exactly the reduced expressions of CFC elements.

\end{proof}

From this Theorem~\ref{thmautomate}, we can deduce some corollaries.
 
 \begin{corollary}
Let $(W,S)$ be a Coxeter system. The generating series of reduced decompositions of CFC elements of $W$ is rational.
 \end{corollary}

\begin{proof}
It is a direct consequence of Theorems~\ref{thmautomate} and \ref{kleene}.
\end{proof}

We can also re-obtain a theorem due to Nadeau (however with the same scheme of proof).

\begin{corollary}
Let $(W,S)$ be a Coxeter system. The set $L$ of reduced decompositions of FC elements is a rational language on the alphabet $S$.
\end{corollary}

\begin{proof}
Let $(W,S)$ be a Coxeter system. We build an automaton  $A$ exactly in the same manner as in the proof of  Theorem~\ref{thmautomate}, with this single difference: the final states are defined as all the states except the sink. By the properties ($ii$)  of Lemma~\ref{lemme1} and ($vi$) of Lemma~\ref{lemme2}, this new automaton recognizes exactly the reduced decompositions of FC elements.
\end{proof}

Let just us remark however that the automaton that we purpose to recognize the reduced decompositions  of $FC$ elements will have in general much more states than the one introduced by Nadeau, as our states encode a lot of informations which are useless to determine if a word is a reduced decompositions of a FC element. 

\subsection{Rationality of $W^{CPC}$}\label{3.2}

If $(W,S)$ is a Coxeter system, we now focus on the nature of the generating series  $W^{CPC}(x)$ defined in the introduction.

For this, we will need the following result, due to Brink and Howlett (1993). We fix on $S$ a total order, which induces a lexicographical order on the set of words  $S^*$. We say that a reduced decomposition  ${\bf w}$ of an element $w \in W$ is \emph{minimal} if for any other reduced decomposition ${\bf w'}$ of $w$, ${\bf w'}$ is greater than ${\bf w}$ in the lexicographical order.

\begin{Theorem}[{\cite[Proposition 3.3 ]{BH}}]\label{anisimov}
Let $(W,S)$ be a Coxeter system, and fix a total order on $S$ and the corresponding lexicographical order on $S^*$. The set of minimal reduced decompositions of elements  $w\in W$ can be recognized by (explicit) finite state automaton.
\end{Theorem}

Using this result, we can now prove Theorem~\ref{thmintro}.

\begin{proof}[Proof of Theorem~\ref{thmintro}]
Let $(W,S)$ be a Coxeter system, and  fix a total order on $S$ and the corresponding lexicographical order on $S^*$.

The language $L_1$ of minimal reduced decompositions of elements $w \in W$ can be recognized by finite state automaton, according to Theorem~\ref{anisimov}. 

The language $L_2$ of reduced decompositions of CFC elements is also recognizable by finite state automaton, according to Theorem~\ref{thmautomate}.
 
 As the languages recognizable by finite state automaton are stable under intersection   (see \cite[Theorem 3.3]{HOP}), the language $L_1 \cap L_2$ is recognizable by finite state automata.  This language contains exactly one reduced decomposition for each CFC element (the minimal one).  According to Theorem~\ref{kleene}, the generating series $W^{CPC}(q)$ is rational.
\end{proof}

We want to end this section by a remark.

\begin{Remark} As the automata designed in the proofs of Theorems~\ref{thmautomate} and \ref{thmintro} are explicit, we are able to implement them both. This gives an algorithmic method, given a Coxeter system $(W,S)$, to compute a regular expression for reduced expression of CFC elements, and a regular expression for minimal reduced expressions of CFC elements. Using this last expression, we can also directly compute the generating series $W^{CFC}(x)$. For example, for finite or affine Coxeter groups, we can re-obtain the results of \cite{PET4}.

\end{Remark}
\section{Open questions and acknowledgements}\label{section4}

Some questions remain open on the two automata defined respectively in the proofs of Theorems~\ref{thmintro} and \ref{thmautomate}. Recall that if $L$ is a regular language, there exists a unique automaton $A$ with a minimal number of states. It seems natural to wonder if our two automata are or not minimal. 

Another question concerns directly the number  of states of these automata. Some tests make for ``small'' Coxeter groups (in the sense that their Coxeter graphs are small) show that these numbers of states are very large. For example, for the Coxeter group $\widetilde{A}_3$ (with four generators), the number of states of the first automaton is 149. One may wonder if we can compute this number or the number of states of the minimal automaton.

\medskip

The author would like to thank P. Nadeau for valuable remarks on the links between rationality and automata theory.

\bibliographystyle{plain}
\bibliography{bibliographieang}

\end{document}